\documentclass{amsart}

\usepackage{helvet, color}

\usepackage{amscd,amsmath,amsxtra,amsthm,amssymb,stmaryrd,xr,mathrsfs,mathtools,enumerate}
\usepackage[all,cmtip]{xy}

 \DeclareFontFamily{U}{wncy}{}
    \DeclareFontShape{U}{wncy}{m}{n}{<->wncyr10}{}
    \DeclareSymbolFont{mcy}{U}{wncy}{m}{n}
    \DeclareMathSymbol{\Sh}{\mathord}{mcy}{"58}

\newtheorem{theorem}{Theorem}[section]
\newtheorem{lemma}[theorem]{Lemma}

\newtheorem{proposition}[theorem]{Proposition}
\newtheorem{corollary}[theorem]{Corollary}
\newtheorem{definition}[theorem]{Definition}

\numberwithin{equation}{section}

\theoremstyle{remark}
\newtheorem{remark}[theorem]{Remark}
\newtheorem{example}[theorem]{Example}

\newcommand{\Gal}{\operatorname{Gal}}

\newcommand{\cor}{\operatorname{cor}}

\newcommand{\fE}{\mathfrak{E}}
\newcommand{\fA}{\mathfrak{A}}

\newcommand{\Zp}{\mathbb{Z}_p}

\newcommand{\Z}{\mathbb{Z}}

\newcommand{\p}{\mathfrak{p}}
\newcommand{\Q}{\mathbb{Q}}

\newcommand{\D}{\mathfrak{D}}

\newcommand{\cL}{\mathcal{L}}

\newcommand{\cO}{\mathcal{O}}

\newcommand{\image}{\mathrm{Im}}

\newcommand{\Hom}{\mathrm{Hom}}
\newcommand{\Sel}{\mathrm{Sel}}
\newcommand{\Char}{\mathrm{char}}

\newcommand{\Gr}{\mathrm{Gr}}

\newcommand{\rank}{\mathrm{rank}}

\newcommand{\ur}{\mathrm{ur}}
\newcommand{\N}{\mathbb{N}}

\newcommand{\X}{\mathcal X}
\newcommand{\tor}{\mathrm{tor}}

\newcommand{\A}{\mathcal{A}}
\newcommand{\Hc}{\mathcal{H}}

\begin{document}

\title[Anticyclotomic Selmer groups of positive coranks -- II]{Comparing anticyclotomic Selmer groups of positive coranks for congruent modular forms -- Part II}

\author[J.~Hatley]{Jeffrey Hatley}
\address[Hatley]{
Department of Mathematics\\
Union College\\
Bailey Hall 202\\
Schenectady, NY 12308\\
USA}
\email{hatleyj@union.edu}

\author[A.~Lei]{Antonio Lei}
\address[Lei]{D\'epartement de math\'ematiques et de statistique\\
Pavillon Alexandre-Vachon\\
Universit\'e Laval\\
Qu\'ebec, QC, Canada G1V 0A6}
\email{antonio.lei@mat.ulaval.ca}

\begin{abstract}
We study the Selmer group associated to a $p$-ordinary newform $f \in S_{2r}(\Gamma_0(N))$ over the anticyclotomic $\mathbb{Z}_p$-extension of an imaginary quadratic field $K/\mathbb{Q}$. Under certain assumptions, we prove that this Selmer group has no proper $\Lambda$-submodules of finite index. This generalizes work of Bertolini in the elliptic curve case. We also offer both a correction and an improvement to an earlier result on Iwasawa invariants of congruent modular forms  by the present authors.
\end{abstract}

\thanks{The second named author's research is supported by the NSERC Discovery Grants Program RGPIN-2020-04259 and RGPAS-2020-00096.}


\maketitle

\section{Introduction}\label{section:intro}

Let $E/\Q$ be an elliptic curve, $K$ a number field, and $p$ a prime. A fruitful way to study the arithmetic of $E$ and its associated $p$-adic Galois representation is to define a Selmer group $\Sel(K_\infty,E)$ associated to $E$ over a $\Z_p$-extension $K_\infty / K$. This Selmer group has the structure of a $\Lambda$-module, where $\Lambda \simeq \Z_p \llbracket X \rrbracket$ is the Iwasawa algebra of the Galois group $\Gal(K_\infty/K)\cong\Zp$.

For many applications, it is useful to know that $\Sel(K_\infty, E)$ has no proper $\Lambda$-submodules of finite index. When $\Sel(K_\infty,E)$ is a cotorsion $\Lambda$-module, this can be proved in a wide variety of contexts; see for instance \cite[$\S$7]{Gr89} and \cite[Theorem 7.4]{kidwell16}. On the other hand, when $\Sel(K_\infty, E)$ is not cotorsion, such as when $K/\Q$ is imaginary quadratic satisfying the Heegner hypothesis and $K_\infty/K$ is the anticyclotomic $\Z_p$-extension, one must use different techniques. One of these techniques was originally developed by Bertolini \cite{Bertolini-Bordeaux} in the case when $E$ has good ordinary reduction at $p$, and it has been recently extended to the supersingular setting by Vigni and the present authors \cite{HLV}.

In \cite{HL-MRL}, the present authors studied the anticyclotomic Iwasawa theory for Selmer groups associated to $p$-ordinary modular forms of even weight. Various auxiliary Selmer groups $\Sel_\mathcal{L}(K_\infty, f)$ were introduced, which were shown to have no proper $\Lambda$-submodules of finite index. By studying the relationship between $\Sel(K_\infty, f)$ and the auxiliary groups $\Sel_\mathcal{L}(K_\infty, f)$, results on the variation of Iwasawa invariants were obtained in the tradition of Greenberg--Vatsal \cite{greenbergvatsal}, Emerton--Pollack--Weston \cite{epw}, and Weston \cite{WestonManuscripta}.

We now summarize the main contributions of this paper:

\begin{itemize}
\item We generalize Bertolini's techniques and results on the non-existence of proper $\Lambda$-submodules of finite index in the Selmer group over $K_\infty$ from the setting of elliptic curves to the setting of modular forms. See Theorem \ref{thm:main} for the statement of our result.

\item We correct an inaccuracy of a result in \cite{HL-MRL} on the structure of finitely generated $\kappa\llbracket X \rrbracket$-modules, where $\kappa$ is a finite field of characteristic $p$. This result was used to compare $\lambda$-invariants of congruent $\Lambda$-modules in op. cit. This inaccuracy was pointed out to us by Meng Fai Lim. See $\S$\ref{sec:structure} and especially Example \ref{example:mistake} and Remark \ref{rmk:error-explanation}.

\item We correct and greatly simplify  the formulae for the $\lambda$-invariants of congruent modular forms given in \cite[Theorems~5.5 and 5.8]{HL-MRL}. Our proof is direct and does not rely on any auxiliary Selmer groups. Most notably, our improved result has no error terms coming from Euler factors or from the cokernel of the localization map. Furthermore, we may remove the condition on $\mu$-invariants by making use of a result of Hsieh \cite{Hsieh-Doc}. See Theorem \ref{thm:variation-of-iwasawa-invariants} for the statement.
\end{itemize}

As we discuss again in Section \ref{sec:nonexistence}, once the appropriate objects have been defined, the techniques used in the present paper and in \cite{HLV} frequently boil down to formal algebraic arguments which are due to Bertolini \cite{Bertolini-Bordeaux}. That is, Bertolini's arguments are valid in broad generality once the appropriate algebraic objects have been defined. In \cite{HLV}, some of Bertolini's arguments have been expanded upon in order to highlight their dependence (or lack thereof) on the various algebraic objects which appear.

\subsection*{Acknowledgements}
We thank Meng Fai Lim for pointing out an inaccuracy in \cite{HL-MRL}, as well as for useful comments and suggestions on an earlier version of this paper. We also thank Ashay Burungale, Ming-Lun Hsieh, Keenan Kidwell and Stefano Vigni for interesting discussions on topics related to this paper. We are grateful to the anonymous referee for offering many useful suggestions which have improved the final version of the paper.

\section{Notation and Assumptions}\label{sec:notation}
Throughout this paper, $p$ denotes a fixed odd prime. Section \ref{sec:structure} is self-contained with its own set of notation. In the rest of the paper, we will consider modular forms $f \in S_k(\Gamma_0(N))$, where $N>3$ is squarefree, $p\nmid N$, and $k \geq 2$ is an even integer. Let $K$ be an imaginary quadratic extension of $\Q$ with discriminant coprime to $Np$ in which $p=\p \bar{\p}$ and the prime divisors of $N$ split. We fix embeddings $K \hookrightarrow \mathbb{C}$ and $\overline{\Q} \hookrightarrow \overline{\Q}_p$.

Let $\mathfrak{F}=\Q_p(\{a_n(f)\})$ denote the finite extension of $\Q_p$ (inside of $\mathbb{C}_p$) generated by the Fourier coefficients of $f$, and let $\cO$ denote its valuation ring. We fix a uniformizer $\varpi$ of $\cO$ and denote its residue field by $\kappa$. When studying two modular forms $f$ and $g$, we will enlarge $\cO$ as necessary so that it contains the coefficients of both  modular forms.

In order to apply results of Bertolini \cite{Bertolini-Compositio}, Longo--Vigni \cite{LongoVigni}, and Chida--Hsieh \cite{ChidaHsieh} we will assume that the triple $(f,K,p)$ is \textit{admissible} in the sense that:
\begin{itemize}
  \item $p$ does not ramify in $\mathfrak{F}$
  \item the $p$-th Fourier coefficient $a_p(f)$ is a unit in $\cO$
  \item $p \nmid 6N(k-2)!\phi(N)h_K$
  \item if $k=2$ then $a_p(f)^2 \not\equiv 1 \mod p$.
\end{itemize}
(Here $\phi$ denotes the Euler totient function and $h_K$ denotes the class number of $K$.)
We will also assume that the residual $p$-adic representation associated to $f$,
\[
\bar{{\rho}}_f \colon G_\Q \to \mathrm{GL}_2(\kappa),
\]
is absolutely irreducible.

 Denote by $K_\infty$ the anticyclotomic $\Z_p$-extension of $K$, and for $n\ge0$ we write $K_n$ for the subextension of $K_\infty$ such that $K_n/K$ is of degree $p^n$. Let $\Lambda$ denote the Iwasawa algebra $\cO[[\Gamma]]$, where $\Gamma=\mathrm{Gal}(K_\infty /K) \simeq \Z_p$. For each $n \geq 1$, write $\Gamma_n=\mathrm{Gal}(K_\infty / K_n)$ and $G_n=\mathrm{Gal}(K_n/K).$

\section{Structure of modules over $\Omega$}\label{sec:structure}
We study the structure of modules over $\Omega$, which will allow us to compare Iwasawa invariants of congruent $\Lambda$-modules. Note that in this section, we are not using the fact that our Iwasawa algebra $\Lambda$ comes from the anticyclotomic $\Zp$-extension of an imaginary quadratic field. All results presented in this section still hold if we replace $\Lambda$ by any Iwasawa algebra that is isomorphic to the power series ring $\cO\llbracket X \rrbracket$.

Given a finitely generated $\Lambda$-module $M$, we define $M_\tor$ to be the maximal torsion submodule of $M$. We recall from \cite[\S1.8 and \S3.1]{Jan} that there is an exact sequence of $\Lambda$-modules
\begin{equation}
   0\rightarrow M_\tor\rightarrow M\rightarrow M^{++}\rightarrow T_2(M)\rightarrow 0,
\label{eq:jan}
\end{equation}
where $M^{++}$ is the reflexive hull of $M$, which is free over $\Lambda$ and $T_2(M)$ is finite.

Recall that $\kappa=\cO/(\varpi)$ denotes the residue field of $\cO$. We write $\Omega$ for the Iwasawa algebra $\kappa[[\Gamma]]=\Lambda/(\varpi)$ over $\kappa$. Since $\kappa$ is a field, it follows that $\Omega$ is a principal ideal domain. In particular, given a finitely generated $\Omega$-module $N$, there is an isomorphism of $\Omega$-modules
\[
N\cong \Omega^{\oplus r}\oplus \bigoplus_{j=1}^t\Omega/(F_j),
\]
where $F_j$ are polynomials in $\Omega$. We define the characteristic ideal of $N$ by
\[
\Char_\Omega(N)=\left(\prod_{j=1}^tF_j\right)\Omega.
\]
We also define $$\lambda(N)=\sum_{j=1}^t\deg(F_j).$$
We note that
\[
|N_\tor|=\left| \bigoplus_{j=1}^t\Omega/(F_j)\right|=|\kappa|^{\lambda(N)}.
\]

\begin{lemma}\label{lem:Omega}
Let $0\rightarrow A\rightarrow B\rightarrow C\rightarrow0$ be a short exact sequence of finitely generated $\Omega$-modules. Suppose that $A$ is torsion over $\Omega$, then
\begin{itemize}
    \item[(i)] $\rank_\Omega B=\rank_\Omega C$;
    \item[(ii)] $\Char_\Omega (B)=\Char_\Omega(A)\Char_\Omega(C)$.
\end{itemize}
\end{lemma}
\begin{proof}
This follows from the same proof as \cite[Proposition~2.1]{HL-MRL}.
\end{proof}

\begin{proposition}\label{prop:correction}
Let $M$ be a finitely generated $\Lambda$-module such that $\mu(M)=0$. Then,
\[
\rank_\Omega(M/\varpi)=\rank_\Lambda M.
\]
Furthermore,
\[
\Char_\Omega(M/\varpi)=\Char_\Omega\left(M_\tor/\varpi\right)\Char_\Omega\left( T_2(M)[\varpi]\right).
\]
\end{proposition}
\begin{proof} Let us write $M'$ for the $\Lambda$-torsion-free quotient $M/M_\tor$.
We begin by considering the following tautological short exact sequence:
\[
0\rightarrow M_\tor\rightarrow M\rightarrow M'\rightarrow0.
\]
Since $M'$ is $\Lambda$-torsion-free, we have $M'[\varpi]=0$. This gives the short exact sequence
\[
0\rightarrow M_\tor/\varpi\rightarrow M/\varpi\rightarrow M'/\varpi\rightarrow0.
\]
Our hypothesis that $\mu(M)=0$ implies that $M_\tor$ is finitely generated over $\cO$. Hence, $M_\tor/\varpi$ is a torsion $\Omega$-module. Therefore, Lemma~\ref{lem:Omega} implies that
\begin{align}
\label{eq:ranks}\rank_\Omega M/\varpi&=\rank_\Omega M'/\varpi;\\
    \label{eq:Omega-tor-free}
    \Char_\Omega(M/\varpi)&=\Char_\Omega(M_\tor/\varpi)\Char_\Omega(M'/\varpi).
\end{align}

Let $r=\rank_\Lambda M$. Consider the exact sequence \eqref{eq:jan} applied to the $\Lambda$-module $M'$. We have
\[
0\rightarrow M'\rightarrow \Lambda^{\oplus r}\rightarrow T_2(M) \rightarrow 0
\]
since $M'_\tor=0$ and $T_2(M)=T_2(M')$. As $\Lambda[\varpi]=0$, we deduce the following exact sequence
\begin{equation}
    0\rightarrow T_2(M)[\varpi]\rightarrow  M'/\varpi\stackrel{\theta}{\longrightarrow} \Omega^{\oplus r}\rightarrow T_2(M) /\varpi\rightarrow 0.
\label{eq:M'jan}
\end{equation}
Recall that $T_2(M)$ is finite. This tells us that $\rank_\Omega M'/\varpi=\rank_\Omega\image\theta=r$. In particular, if we combine this with \eqref{eq:ranks}, we see that $\rank_\Omega M/\varpi=r$. Furthermore, $\Omega^{\oplus r}$ is $\Omega$-torsion-free. Thus, $\image\theta\cong\Omega^{\oplus r}$ as $\Omega$-modules. Therefore, we deduce from \eqref{eq:M'jan} a short exact sequence
\[
 0\rightarrow T_2(M)[\varpi]\rightarrow  M'/\varpi\stackrel{\theta}{\longrightarrow} \Omega^{\oplus r}\rightarrow 0.
\]
 Lemma~\ref{lem:Omega} now tells us that
 \[
 \Char_\Omega M'/\varpi=\Char_\Omega T_2(M)[\varpi].
 \]
On combining this with \eqref{eq:Omega-tor-free}, our result follows.
 \end{proof}

 Proposition~\ref{prop:correction} leads us to introduce the following definition.
 \begin{definition}\label{def:error-term}
 Let $M$ be a finitely generated module over $\Lambda$. We define  $ c(M)$ to be the unique integer satisfying the equation
 \[
|T_2(M)[\varpi]|=|\kappa|^{c(M)}.
 \]
 \end{definition}

\begin{corollary}\label{cor:lambda-error-equality}
Suppose that $A$ and $B$ are two finitely generated $\Lambda$-modules such that
\begin{itemize}
    \item[(i)]  $A/\varpi\cong B/\varpi$ as $\Omega$-modules;
    \item[(ii)] $\mu(A)=\mu(B)=0$;
    \item[(iii)] Both $A$ and $B$ admit no non-trivial finite $\Lambda$-submodules.
\end{itemize}
Then, we have
\begin{itemize}
    \item[(a)]$\rank_\Lambda A=\rank_\Lambda B$;
    \item[(b)]$\lambda(A)+c(A)=\lambda(B)+c(B)$.
\end{itemize}
\end{corollary}
\begin{proof}
Our hypothesis (ii) allows us to apply Proposition~\ref{prop:correction} to both $A$ and $B$. Thus, the hypothesis (i) tells us that
\[
\rank_\Lambda A=\rank_\Omega A/\varpi=\rank_\Omega B/\varpi=\rank_\Lambda B
\]
and that
\begin{align*}
    \Char_\Omega\left(A_\tor/\varpi\right)\Char_\Omega&\left( T_2(A)[\varpi]\right)=\Char_\Omega(A/\varpi)\\
    &=\Char_\Omega(B/\varpi)=\Char_\Omega\left(B_\tor/\varpi\right)\Char_\Omega\left( T_2(B)[\varpi]\right).
\end{align*}

But $M_\tor\cong \cO^{\lambda(M)}$ as $\cO$-modules for both $M=A$ and $B$ thanks to  (ii) and (iii). Consequently,
\[
M_\tor/\varpi\cong \kappa^{\lambda(M)}
\]
as  $\kappa$-vector spaces. Hence, our result follows.
\end{proof}

\begin{example}\label{example:mistake}
We thank Meng Fai Lim for bringing to our attention the following example. Let $A=\Lambda\oplus \cO$ and $B$ the maximal ideal of $\Lambda$. Note that neither $A$ nor $B$ admits non-trivial finite $\Lambda$-submodules and that $\mu(A)=\mu(B)=0$. Furthermore, consider the short exact sequence
\[
0\rightarrow B\rightarrow \Lambda\rightarrow\kappa\rightarrow0.
\]
Since $\Lambda[\varpi]=0$, we obtain the following exact sequence:
\[
0\rightarrow\kappa\rightarrow B/\varpi\stackrel{\theta}{\longrightarrow} \Omega\rightarrow\kappa\rightarrow0.
\]
Since both the kernel and cokernel of $\theta$ are isomorphic to $\kappa$, which is of rank $0$ over $\Omega$, we see that the image of $\theta$ is free of rank one over $\Omega$. This then gives the $\Omega$-isomorphism
\[
B/\varpi\cong \Omega\oplus\kappa\cong A/\varpi.
\]
In particular, the three hypotheses in Corollary~\ref{cor:lambda-error-equality} hold.

It is clear that $\rank_\Lambda A=\rank_\Lambda B=1$, confirming property (a) of the corollary. We have $T_2(A)=0$ and $T_2(B)=\kappa$. Thus, $c(A)=0$ and $c(B)=1$. It is not hard to see that $\lambda(A)=1$ and $\lambda(B)=0$.  Thus,
\[
\lambda(A)+c(A)=\lambda(B)+c(B)=1,
\]
confirming property (b). From this example, we see that  $\lambda(A)$ and $\lambda(B)$ can be different and that the terms $c(A)$ and $c(B)$ are indispensable in (b).
\end{example}

\begin{remark}\label{rmk:error-explanation}
The preceding results correct some inaccuracies in \cite[Lemmas 2.8 and 2.9]{HL-MRL} which were brought to our attention by Meng Fai Lim.

\begin{itemize}

\item In Lemma 2.8, the condition  $(C/\varpi)^{\Gamma_n} =0$ is equivalent to saying that $C=0$. To see this, note that $C$ is finite, and so it is both a discrete and compact module. By the $p$-group fixed point theorem, we have $(C/\varpi)^{\Gamma_n} =0$ if and only if $C/\varpi=0$. The latter holds if and only if $C=0$ by Nakayama's lemma, which is valid since C is compact.

\item Lemma 2.9 will not hold if $P$ is finite. Since $\Gamma_n$ will act trivially on $P(i)$ for $n \gg 0$, the conclusion can never be attained. The problem in the proof lies in second-to-last line: each summand in the direct sum {\it a priori} maps to $P(i)$ rather than being a submodule, but upon taking their sum, these maps may introduce a kernel due to the presence of finite modules.
\end{itemize}
Thus, the results from this section of the present paper should be used in lieu of those mentioned above.
\end{remark}

\section{Comparing $\lambda$-invariants of congruent modular forms}\label{sec:comparing-invariants}
Recall from Section \ref{sec:notation} that $f\in S_k(\Gamma_0(n))$ denotes a modular form of even weight $k=2r$ for which the triple $(f,K,p)$ is admissible.

Denote by $\rho_f$ the associated $G_\Q$-representation constructed in \cite{nekovar1992}; it is realized by a free ${\mathcal O}$-module of rank $2$ which we denote by $T$, and $V=T \otimes \mathfrak{F}$ is the $r$-th Tate twist of the Galois representation constructed by Deligne. Finally, let $A = V/T$. Recall that we assume $T / \varpi T$ is irreducible, hence $T$ is unique up to scaling. Our choice of normalization makes both $V$ and $T$ self-dual.

For a rational prime $\ell$, write $G_\ell$ for the decomposition group at $\ell$. As described in \cite[$\S$1.2]{ChidaHsieh}, the representation $\rho_f$ has the following local properties:
\begin{equation}\label{eq:local-at-p}\rho_f \vert_{G_p} \sim \left(
\begin{matrix}
\chi_p^{-1}\epsilon^{r} & \ast \\
0 & \chi_p \epsilon^{1-r}
\end{matrix}
\right)\end{equation}
where $\chi_p$ is unramified and $\epsilon$ is the $p$-cyclotomic character, and
for $\ell \mid N$, \begin{equation}\label{eq:local-at-ell}\rho_f \vert_{G_\ell} \sim \left(
\begin{matrix}
\pm \epsilon & \ast \\
0 & \pm 1
\end{matrix}
\right)\end{equation}
where $\ast$ is ramified. This can be deduced via the local Langlands correspondence; see for instance \cite[$\S$3.3]{WestonUnob}. Note that we are using the assumption that $N$ is squarefree.

For every integer $m \geq 1$, let $A_m=A[\varpi^m]$ denote the $\varpi^m$-torsion submodule of $A$. There is a canonical isomorphism between $A_m$ and $T_m=T/\varpi^mT$. We will sometimes find it convenient to denote $A_\infty = A$.

In this section, we first recall the various definitions of Selmer groups over anticyclotomic extensions following \cite{Gr89}, and then we  present corrected (and improved) versions of our results in \cite[\S5]{HL-MRL}.

\subsection{Anticyclotomic Selmer Groups}\label{sec:Selmer-definitions}

Since $f$ is $p$-ordinary and since we are working with the $r$-th Tate twist of the associated Galois representation, there exists a unique $G_{\Q_p}$-invariant line $\mathcal{F}^+V \subset V$ on which inertia $I_{\Q_p}$ acts by $\epsilon^r$, where $\epsilon$ is the $p$-cyclotomic character. Let $\mathcal{F}^+ A$ be the image of $\mathcal{F}^+V$ under the natural projection map $V \rightarrow A$, and for each $m \geq 1$ set $\mathcal{F}^+ A_m = \mathcal{F}^+ A \cap A_m$. We define $\mathcal{F}^+ T_m$ in a similar way. Then for $W \in \{ A_m, T_m\}$, we denote by $\mathcal{F}^-W=W/\mathcal{F}^+W$ the unramified quotient.

Let $E$ be any finite extension of $K$ and let $m \in \mathbf{N} \cup \{\infty\}$. For any place $v$ of $E$, define the ordinary local condition
\[
H^1_f (E_v, A_m) = \begin{cases} \mathrm{ker}\left( H^1(E_v,A_m) \rightarrow {H^1(E_v^{\ur},A_m)} \right), & v \nmid p, \\
\mathrm{ker}\left( H^1(E_v,A_m) \rightarrow H^1(E_v,\mathcal{F}^-A_m) \right), & v \mid p,
\end{cases}
\]
where $E_v^{\mathrm{un}}$ denotes the maximal unramified extension of $E_v$.

Recall that $p=\mathfrak{p} \bar{\mathfrak{p}}$ splits in $K$. Following \cite[Definition 2.2]{Castella}, for $v \mid p$ and $\mathcal L_v \in \{ \emptyset, \Gr, 0 \}$, set
\[
H^1_{\mathcal L_v}(E_v,A_m) = \begin{cases}
H^1(E_v,A_m) & \text{if}\ \mathcal L_v = \emptyset, \\
H^1_f(E_v,A_m) & \text{if}\ \mathcal L_v = \Gr, \\
\{0\} & \text{if}\ \mathcal L_v = 0.
\end{cases}
\]

If $\Sigma$ is a finite set of places, we denote by $E_\Sigma$ the maximal extension of $E$ unramified outside the places above $\Sigma$, and we write $H^i_\Sigma(E,*)$ for the Galois cohomology $H^i({E_\Sigma/E},*)$.

\begin{definition}
Let $m \in \mathbf{N} \cup \{\infty\}$. Then for any finite extension $E$ of $K$, a set $\mathcal L = \{ \mathcal L_v \}_{v \mid p}$ of local conditions at $p$, and for any finite set $\Sigma$ of places of $E$ containing those which divide $Np\infty$, we define the Selmer group
\[
\Sel_\mathcal{L}(E,A_m)= \ker \left( H^1_\Sigma(E,A_m) \rightarrow \prod_{\substack{v \in \Sigma\\ v \nmid p}} \frac{H^1(E_v,A_m)}{H^1_f(E_v,A_m)} \times \prod_{v \mid p} \frac{H^1(E_v,A_m)}{H^1_{\mathcal{L}_v}(E_v,A_m)}\right).
\]

We then define
\begin{align*}
\Sel_\mathcal{L}(K_\infty, A_m) &= \varinjlim_{n} \Sel_\mathcal{L}(K_n,A_m),
\end{align*}
where the limit is taken with respect to the natural restriction map.

Whenever $\mathcal{L}=\{\Gr,\Gr\}$, we omit the subscript and just write $\Sel(K_n,A_m)$.
\end{definition}


\begin{remark}\label{rmk:nonsurj} It is very often helpful to show that the global-to-local cohomology map which defines a Selmer group is surjective; for instance, this is the problem which is studied by Greenberg in \cite{Greenberg11}, where this surjectivity is used to deduce the non-existence of finite-index $\Lambda$-submodules. However, the global-to-local map defining our $\Sel(K_\infty,A)$ is \textit{not} surjective; nevertheless, we deduce that it has no finite-index submodules in Section \ref{sec:nonexistence}.  \end{remark}

Let us set some notation for the rest of the paper. For any $1\leq n \leq \infty$ we take $\Sigma_n$ to be the set of places of $K_n$ above the rational primes dividing $Np \infty$, and we simply write $\Sigma$ instead of $\Sigma_\infty$.

 We note that the set $\Sigma$ is finite since we have assumed that every prime dividing $Np$ splits in $K$, and since $K_\infty/K$ is anticyclotomic, no rational prime splits completely in $K_\infty$ \cite[Corollary 1]{Brink}.

Recall that $A_m=A[\varpi^m]$ and that $A$ is a divisible $\cO$-module.

\begin{lemma}\label{lem:pre-control}
Let $m, n \in \mathbb{N} \cup \{\infty\}$. We have isomorphisms
\begin{align}
H^1(K,A_m) &\simeq H^1(K_n,A_m)^{G_n}, \quad \text{and}\label{eq:isom1}\\
H^1(K,A_m) &\simeq H^1(K,A)[\varpi^m].\label{eq:isom2}
\end{align}
\end{lemma}
\begin{proof}
The isomorphism \eqref{eq:isom1} follows from the inflation-restriction exact sequence and our assumption that $\bar{\rho}_f$ is absolutely irreducible. To prove $\eqref{eq:isom2}$, consider the tautological exact sequence
\begin{equation}\label{eq:tautological}
0 \rightarrow A_m \rightarrow A \xrightarrow{\varpi^m} A \rightarrow 0.
\end{equation}
Taking the long exact sequence in $G_K$-cohomology we obtain the exact sequence
\[
H^0(K,A) \rightarrow H^1(K,A_m) \rightarrow H^1(K,A) \xrightarrow{\varpi^m} H^1(K,A).
\]
Once again, our assumption that $\bar{\rho}_f$ is irreducible implies that the first term is zero. This concludes the proof.
\end{proof}

We now prove the following control theorem. We mention that, in the case $\mathcal{L}=\{\Gr,\Gr\}$, this is essentially \cite[Proposition 1.9(1)]{ChidaHsieh}, taking $S=\Delta=1$ in their notation.

\begin{theorem}\label{thm:control}
Let $m,m',n,n' \in \N \cup \{ \infty\}$ with $m \leq m'$ and $n
\leq n'$. There is an isomorphism
\[
\mathrm{res}_{m,n} \colon \Sel_{\mathcal{L}}(K_n, A_m) \to \Sel_{\mathcal{L}}(K_{n'},A_{m'})[\varpi^m]^{\mathrm{Gal}(K_{n'} / K_n)}.
\]
\end{theorem}
\begin{proof}
Let us first show that restriction gives us an isomorphism
\[
\Sel_{\mathcal{L}}(K,A_m) \rightarrow \Sel_{\mathcal{L}}(K_n,A_m)^{G_n}.
\]
In light of \eqref{eq:isom1} from the previous lemma, this will follow from showing the injectivity of the maps
\[
\begin{cases}
H^1(K_v^\ur,A_m) \rightarrow {H^1(K_{n,v}^\ur,A_m)}, & v \nmid p, \\
H^1(K_v,\mathcal{F}^-A_m) \rightarrow H^1(K_{n,v},\mathcal{F}^-A_m), & v \mid p,\\
H^1(K_v,A_m) \rightarrow H^1(K_{n,v},A_m), & v \mid p.
\end{cases}
\]
(Here and at other times we abuse notation and write $v$ for a compatible pair of places above the same rational prime in both $K$ and $K_n$.)

For the case when $v \nmid p$, note that since $K_\infty/K$ is an anticyclotomic $\Z_p$-extension, $v$ divides a rational prime $\ell$ which is not ramified at $K_\infty$, so $K_v^\ur = K_{n,v}^\ur$. For $v \mid p$, \cite[Lemma 1.8]{ChidaHsieh} says that $H^0(K_{n,v},\mathcal{F}^-A)=0$, and it is clear from the local description \eqref{eq:local-at-p} that $H^0(K_{n,v},A)=0$, so injectivity of the second and third maps follows from the inflation-restriction exact sequence; note that in the case $k=2$, this uses the admissibility assumption $a_p(f)^2 \not\equiv 1 \mod p$.

Now let us prove that we have an isomorphism
\[
\Sel(K,A_m) \rightarrow \Sel(K,A)[\varpi^m].
\]
The theorem will then follow from the composition of these various isomorphisms.  In light of \eqref{eq:isom2}, it suffices to show the injectivity of the maps
\[
\begin{cases}
H^1(K_v^\ur,A_m) \rightarrow {H^1(K_v^\ur,A)}, & v \nmid p, \\
H^1(K_v,\mathcal{F}^-A_m) \rightarrow H^1(K_v,\mathcal{F}^-A), & v \mid p\\
H^1(K_v,A_m) \rightarrow H^1(K_v,A), & v \mid p.
\end{cases}
\]
Taking the long-exact sequence corresponding to the tautological exact sequence \eqref{eq:tautological}, we see that the kernels of the maps above are given respectively by
\[
\begin{cases}
A^{I_v}/\left(\varpi^m A^{I_v}\right) & v \nmid p, \\
H^0(K_v,\mathcal{F}^-A) & v \mid p, \ \text{or}\\
H^0(K_v,A) & v \mid p.
\end{cases}
\]
where we have written $I_v$ for the inertia group at $v$. Thus, the case when $v=p$ follows by the same argument that we gave above.

Now consider $v \nmid p$. If $A$ is unramified at $v$, then  $A^{I_v}=A$ is divisible, we see that $A^{I_v}/\varpi^m A^{I_v}=0$ as desired. Otherwise, suppose $v \mid N$ and $A$ is ramified at $v$. In this case,  by \eqref{eq:local-at-ell}, the local representation restricted to inertia has the form
\[
\rho_f \big\vert_{I_v} \sim \left( \begin{matrix}\pm 1 & \ast \\ 0 & \pm 1 \end{matrix}\right)
\]
with $\ast$ nontrivial. Thus $A^{I_v}$ is either $0$ or isomorphic to $F_p / \cO$, in which case it is once again divisible, and this finishes the proof.
\end{proof}
\subsection{Variation of Iwasawa invariants}

In this section we give an application of the results in Section \ref{sec:nonexistence}. Throughout this section, assume that we have two modular forms $f$ and $g$, of levels $N_f$ and $N_g$ respectively (possibly of different weights). Let $K / \Q$ be an imaginary quadratic field and $p$ a prime such that the triples $(f,K,p)$ and $(g,K,p)$ are both admissible in the sense of \S\ref{sec:notation}, and assume further that
\[
\bar{\rho}_f \simeq \bar{\rho}_g.
\]
So in particular, modifying our previous notation to handle two modular forms by writing $A_f$ and $A_g$, we have an isomorphism of $G_\Q$-modules
\begin{equation}\label{eq:isom-mod-p}
A_f[\varpi] \simeq A_g[\varpi].
\end{equation}
We expand our sets of places $\Sigma_n$ to include the places above the rational primes dividing $N_g$.

We now wish to begin studying the Iwasawa invariants of our Selmer groups. For each pair of local conditions $\mathcal{L}$ defined in $\S$\ref{sec:Selmer-definitions}, denote by $\mathcal{X}_{\mathcal{L}}(f)$ the Pontryagin dual of the Selmer group $\Sel_{\mathcal{L}}(K_\infty,A_f)$, and similarly for $g$. Write
\[
\mu_{\mathcal{L}}(f) = \mu(\mathcal{X}_{\mathcal{L}}(f)) \quad \text{and} \quad \lambda_{\mathcal{L}}(f) = \lambda(\mathcal{X}_{\mathcal{L}}(f))
\]
as defined in Section \ref{sec:structure}, and similarly for $g$. Recall the notation from Definition \ref{def:error-term}, and set
\[
c_{\mathcal{L}}(f) = c(\mathcal{X}_{\mathcal{L}}(f)),
\]
and similarly for $g$. In all instances, we omit ${\mathcal{L}}$ from the notation when ${\mathcal{L}}=\{\Gr,\Gr\}$.

As explained in Example \ref{example:mistake}, there is an inaccuracy in the statements of \cite[Lemmas 2.8 and 2.9]{HL-MRL} due to the omission of the error terms $c_\cL(f)$ and $c_\cL(g)$. This in turn leads to an inaccuracy in the statement of \cite[Theorems~5.5 and 5.8]{HL-MRL}. In what follows, we will both correct and significantly improve these two theorems. On the one hand,  we  correct them by including the terms $c_\cL(f)$ and $c_\cL(g)$. On the other hand, we improve them by using Theorem \ref{thm:control} to completely bypass the use of non-primitive or auxiliary Selmer groups, thus eliminating all  the ``error terms'' in the statements of \cite[Theorems~5.5 and 5.8]{HL-MRL}. Furthermore, we remove the condition  on $\mu$-invariants utilizing the following theorem.

\begin{theorem}\label{thm:mu-vanishes}
Suppose $(f,K,p)$ is admissible in the sense of $\S$2. 
If $$\mathcal{L} \in \{\{\emptyset,0\}, \{\Gr,0\}, \{\emptyset,\Gr\}, \{\Gr,\Gr\} \},$$
then $\mu_{\mathcal{L}}(f)=0$.
\end{theorem}
\begin{proof} 
Recall from $\S$\ref{sec:Selmer-definitions} the definitions of the various Selmer groups $\Sel_{\mathcal{L}}(K_\infty,A)$ with different local conditions at the primes above $p$, and write $\mathcal{X}_{\mathcal{L}}(f)$ for their Pontryagin duals. Let $\mathcal{L}(f/K)$ denote the Bertolini--Darmon--Prasanna $p$-adic $L$-function attached to $f$ over $K$ defined in \cite{BDP}. By \cite[Theorem 1.5]{KobOta}, we have
\[
\mathrm{char}_\Lambda\left( \mathcal{X}_{\emptyset,0}(f) \right) \supseteq \left( \mathcal{L}(f/K) \right)^2.
\]
By \cite[Theorem B]{Hsieh-Doc}, we know $\mu(\mathcal{L}(f/K))=0$, hence $\mu(\mathcal{X}_{\emptyset,0}(f))=0.$

The Poitou-Tate sequence gives us a surjective map
\[
\mathcal{X}_{\emptyset,0}(f) \twoheadrightarrow \mathcal{X}_{\Gr,0}(f),
\]
and since each of these groups is $\Lambda$-torsion by \cite[Corollary 3.8]{HL-MRL}, the kernel of this map must also be torsion. Since $\mu(\mathcal{X}_{\emptyset,0}(f))=0$ and $\mu$-invariants are nonnegative, we deduce that $\mu(\mathcal{X}_{\Gr,0}(f))=0$ by \cite[Proposition 2.1]{HL-MRL}. By \cite[Lemma~3.6]{HL-MRL}, this in turn allows us to deduce that $\mu(\mathcal{X}_{\emptyset,\Gr}(f))=0.$

As shown in the proof of \cite[Theorem 5.8]{HL-MRL}, we have an exact sequence
\[
0 \rightarrow C \rightarrow \mathcal{X}_{\emptyset,\Gr}(f) \rightarrow \mathcal{X}(f) \rightarrow 0,
\]
where $C$ is a torsion $\Lambda$-module. Applying \cite[Proposition 2.1]{HL-MRL} once again yields $\mu(f)=0$.
\end{proof}

We now offer a corrected and improved version of \cite[Theorems~5.5 and 5.8]{HL-MRL}.

\begin{theorem} \label{thm:variation-of-iwasawa-invariants}
Let $f$ and $g$ be modular forms such that $(f, K, p)$ and $(g,K,p)$ are admissible in the sense of $\S$\ref{sec:notation} and for which there is an isomorphism
$
\bar{\rho}_f \simeq \bar{\rho}_g
$
between their (irreducible) residual Galois representations.
Choose a pair of local conditions
\[
\mathcal{L} \in \{ \{\Gr,\Gr\}, \{\Gr,\emptyset\}, \{\emptyset,0\} \}.
\]
If $\mathcal{L}=\{\Gr,\Gr\}$, assume that $\Sel_{\mathcal{L}}(K_\infty,A_f)$ and $\Sel_{\mathcal{L}}(K_\infty,A_g)$ admit no proper $\Lambda$-submodules of finite index. Then we have
\[
\lambda_{\mathcal{L}}(f) + c_{\mathcal{L}}(f) = \lambda_{\mathcal{L}}(g) + c_{\mathcal{L}}(g).
\]
\end{theorem}

\begin{proof}
Our assumption that $A_f[\varpi] \simeq A_g[\varpi]$ implies, thanks to Theorem \ref{thm:control}, that we have isomorphisms
\[
\Sel_{\mathcal{L}}(K_\infty,A_f)[\varpi] \simeq \Sel_{\mathcal{L}}(K_
\infty,A_g)[\varpi],
\]
and taking duals, this gives an isomorphism \[
\mathcal{X}_{\mathcal{L}}(f)/\varpi \simeq \mathcal{X}_{\mathcal{L}}(g)/\varpi.
\]

By Theorem \ref{thm:mu-vanishes}, we have
\[
\mu_{\mathcal{L}}(f)= \mu_{\mathcal{L}}(g)=0.
\]
If ${\mathcal{L}}=\{\Gr,\Gr\}$, then by Theorem \ref{thm:main} we know that $\mathcal{X}_{\mathcal{L}}(f)$ and $\mathcal{X}_{\mathcal{L}}(f)$ have no proper $\Lambda$-submodules of finite index; for the other choices of $\mathcal{L}$, this is \cite[Proposition 3.12]{HL-MRL}. Thus, all of the hypotheses of Corollary \ref{cor:lambda-error-equality} are satisfied, and the desired result follows.
\end{proof}

\begin{remark}
In the case where the Selmer groups $\Sel(K_\infty,A_f)$ and $\Sel(K_\infty,A_g)$ are $\Lambda$-cotorsion, the analogue of Theorem~\ref{thm:variation-of-iwasawa-invariants} has been studied in \cite{CastKimLon}.
\end{remark}
\begin{remark}
Note that when $\cL=\{\emptyset,0\}$, we have $c_\cL(f)=c_\cL(g)=0$ since $\mathcal{X}_\cL(f)$ and $\mathcal{X}_\cL(g)$ are $\Lambda$-torsion.
When $\mathcal{X}_\cL(f)$ and $\mathcal{X}_\cL(g)$ are not $\Lambda$-torsion, it is not clear to us how to calculate $c_{\mathcal{L}}(f)$ and $c_{\mathcal{L}}(g)$ in general since it requires an explicit description of the $\Lambda$-modules $\mathcal{X}_\cL(f)$ and $\mathcal{X}_\cL(g)$. But it seems reasonable to expect that $c_{\mathcal{L}}(f)=c_{\mathcal{L}}(g)$ if $f$ and $g$ belong to the same branch of a Hida family, which would imply that $\lambda_{\mathcal{L}}(f)=\lambda_{\mathcal{L}}(g)$, mirroring the results in \cite{CastKimLon}. In certain special cases, it is possible to determine the structure of $\mathcal{X}(f)$ and $\mathcal{X}(g)$ via the control theorem. For example, if $\Sel(K,A_f)^\wedge\cong \cO$, then $\Sel(K_\infty,A_f)^\wedge\cong\Lambda$ and $c(f)=0$ (see the discussion in \cite[Page 428]{Perrin-Riou}).
\end{remark}

\section{Non-existence of submodules of finite index}\label{sec:nonexistence}

The goal of this section is to prove the following theorem. \noindent Hypothesis \textbf{(Sha)} is defined in $\S$\ref{subsec:universal-norm-com[utations}.

\begin{theorem}\label{thm:main}
Assume $(f,K,p)$ is admissible in the sense of $\S\ref{sec:notation}$ and that hypothesis \textbf{(Sha)} hold. Then $\Sel(K_\infty,A)$ contains no proper $\Lambda$-submodules of finite index.
\end{theorem}

Suppose for the moment that $f \in S_2(\Gamma_0(N))$ is the modular form associated to a $p$-ordinary elliptic curve $E/\Q$. In this case, Theorem \ref{thm:main} is due to Bertolini \cite[Theorem 7.1]{Bertolini-Bordeaux}. Thus, we aim to extend this result to the more general class of modular forms which we have been studying in this paper.

Bertolini's result was recently extended to the case of elliptic curves with supersingular reduction at $p$ in \cite{HLV} by Vigni and the present authors. The careful reader will notice that, once certain objects have been properly defined and their properties established, the proof of \cite[Theorem 7.1]{Bertolini-Bordeaux} goes through quite formally. Our strategy in the present paper will thus be to put all of the necessary pieces into place and then describe the main steps of the proof, referring the reader to the proofs in \cite[Section 7]{Bertolini-Bordeaux} and \cite[Section 5]{HLV} for more details.

\subsection{Heegner cycles}\label{subsec:heegner-cycles} Recall that by work of  Bertolini--Darmon--Prasana \cite{BDP} (see also \cite{CastellaHsieh, howard,LongoVigni}), one may use generalized Heegner cycles in order to construct a system of compatible cohomology classes which aid in the study of our Selmer groups.

Recall that $T_m=T/\varpi^m T$ and $A_m = A[\varpi^m]$, and in fact these are isomorphic, with the notation suggesting compatibility with projective and direct limits, respectively. We define, for $n \in \N$,
\[
\Sel(K_n, T) = \varprojlim_{m} \Sel(K_n,A_m),
\]
and then we define
\[
\hat{S}(K_\infty,T)= \varprojlim_n \Sel(K_n,T),
\]
where the limits are taken with respect to the multiplication-by-$\varpi$ and corestriction maps.

In what follows, we shall write $M^\wedge$ for the Pontryagin dual of a $\cO$-module $M$.

\begin{lemma}
There is a canonical isomorphism of $\Lambda$-modules
\[
\hat{S}(K_\infty, T) \simeq \Hom_\Lambda(\Sel(K_\infty, A)^\wedge, \Lambda).
\]
\end{lemma}
\begin{proof}
The argument of \cite[Lemme~5]{Perrin-Riou} holds in our setting upon replacing the control theorem used in \cite{Perrin-Riou} with Theorem~\ref{thm:control}.
\end{proof}

\begin{definition}\label{def:heegner-points}
For each $n \geq 1$, we denote by $z_n$ the generalized Heegner class
\[
z_n \in \Sel(K_n, T) 
\]
constructed in \cite[$\S$4.4]{CastellaHsieh}. 
\end{definition}

These Heegner classes satisfy the following ``three-term relation'':
\begin{equation}\label{eq:three-term-relation}
\mathrm{cores}_{K_{n+1}/K_n} (z_{n+1}) = a_p(f) z_n - p^{k-2} \mathrm{res}_{K_{n+1}/K_{n}}(z_{n-1}).
\end{equation}
Here, $\mathrm{cores}_{K_{n+1}/K_n}$ and $\mathrm{res}_{K_{n+1}/K_n}$ denote the natural corestriction and restriction maps at the level of Galois cohomology.

Let $R_{m,n}=(\cO/\varpi^m\cO)][G_n]$, where $G_n=\mathrm{Gal}(K_n/K)$.

\begin{definition}\label{def:heegner-modules}
Denote by $\mathcal{A}_{m,n}$ the $R_{m,n}$-submodule of $\Sel(K_n,A_m)$ generated by the Heegner class $z_n$. Taking limits, we obtain the $\Lambda$-module
\[
\mathcal{A}_\infty = \varinjlim_{m} \mathcal{A}_{m,m} \subset \Sel(K_\infty,A)
\]
and its Pontryagin dual
\[
\mathcal{H}_\infty =\mathcal{A}_\infty^\wedge =\varprojlim_{m} \left(\mathcal{A}_{m,m}\right)^\wedge \subset \hat{S}(K_\infty,T).
\]
We also define
\[
\fA_{m,n}=\mathcal{A}_{\infty}^{\Gamma_n}[\varpi^m].
\]
When $(m,n)=(1,n)$ we shall omit $m$ from the notation and simply write $R_n, \mathcal{A}_n,$ and $\fA_n$.
\end{definition}

The following lemma will allow us to henceforth consider $\A_{m,n}$ a $R_{m,n}$-submodule of $\Sel(K_n,T_m)$.

\begin{lemma}\label{lem:heegner-in-selmer}
For each $m,n$ there exist injections of $R_{m,n}$-modules
\[
\A_{m,n} \hookrightarrow \Sel(K_n,T_m)
\]
and
\[
\A_{m,n} \hookrightarrow \A_{m,n+1}.
\]
The image of the latter map is precisely $\A_{m,n+1}^{\mathrm{Gal}(K_{n+1}/K_n)}$.
\end{lemma}
\begin{proof}
Remembering that $A_m \simeq T_m$, this is immediate from Definitions \ref{def:heegner-points} and \ref{def:heegner-modules} and Theorem \ref{thm:control}.
\end{proof}

In \cite{Bertolini-Bordeaux}, Bertolini imposed the hypothesis that $\mathcal{A}_\infty \neq 0$, but this is now known to be true, as we record in the next result.

\begin{proposition}\label{prop:heegner-non-triv}
For $n \gg 0$, $\A_n \neq 0$. In particular, $\mathcal{A}_\infty \neq 0$.
\end{proposition}
\begin{proof}
The $p$-adic logarithm of the generalized Heegner cycles $z_n$ interpolate special values of the Bertolini--Darmon--Prasanna $p$-adic $L$-function $\mathcal{L}(f/K)$ by \cite[Theorem 4.9]{CastellaHsieh}. Thus, the result follows from \cite[Theorem C]{Hsieh-Doc}.
\end{proof}

\begin{remark} In \cite{Hsieh-Doc}, Hsieh works with imaginary quadratic fields $K/\Q$ satisfying the same Heegner hypothesis as in the present paper. The non-triviality of Heegner cycles mod $p$ has also been studied by Burungale in the setting where the generalized Heegner hypothesis is satisfied; see e.g. \cite{ashay2}.
\end{remark}

The proof of the next proposition is similar to that of \cite[Proposition 4.7]{LongoVigni2}.

\begin{proposition}
The $R_n$-module $\fA_n$ is cyclic.
\end{proposition}

\begin{proof}
Let $n \gg 0$ so that $\A_{n,n} \neq 0$. Then since it is a cyclic $R_n$-module, we have a natural surjection $R_{n,n} \rightarrow \A_{n,n}$ so by duality we have an injection
\[
\A_{n,n}^\wedge \hookrightarrow R_{n,n}^\wedge  \simeq (\cO / \varpi^n \cO)[G_n].
\]

Taking limits, we obtain an injection $\Hc_\infty \hookrightarrow \Lambda$. Since $\Lambda$ is Noetherian, this implies $\Hc_\infty$ is a finitely-generated and torsion-free $\Lambda$-module of rank at most $1$. Since $\Lambda$ is a domain, torsion-free implies free, and our assumption that $\A_\infty \neq 0$ then implies that $\Hc_\infty$ has rank $1$.

Thus,
\begin{align*}
\fA_n &= \A_\infty^{\Gamma_n}[\varpi]\\
&= \left(\Hc_\infty^\wedge \right)^{\Gamma_n}[\varpi] \\
&\simeq \left(\Hc_\infty / \varpi\right)_{\Gamma_n} \\
&\simeq \left(\cO / \varpi \cO \right)[\Gamma]_{\Gamma_n}\\
&\simeq R_n
\end{align*}
which proves the claim.
\end{proof}

\subsection{Pairings and universal norms}

For each $m,n \in \N$ write $I_{m,n}$ for the augmentation ideal of $R_{m,n}$. The following result generalizes \cite[Proposition~6.3]{Bertolini-Bordeaux} to our current setting.

\begin{proposition} \label{prop:tatepairing}
Let $m,n\in\N$. There is a perfect pairing of Tate cohomology groups
\[
{\langle\cdot,\cdot\rangle}_{m,n} : \hat{H}^0\bigl(G_n,\Sel(K_n,A_m)\bigr)\times \hat{H}^{-1}\bigl(G_n,\Sel(K_n,A_m)\bigr)\longrightarrow I_{m,n}/I_{m,n}^2.
\]
\end{proposition}
\begin{proof}
This proof is almost entirely formal, relying on arguments regarding abstract Galois cohomology, as in the proofs of \cite[Proposition~6.3]{Bertolini-Bordeaux} and \cite[Proposition 5.1]{HLV}. One generalizes from the elliptic curve setting to the modular form setting by using the appropriate version of local Tate duality as described in \cite[\S1.3]{ChidaHsieh}. The only part that is not entirely formal is the extension of \cite[Theorem 3.2]{BD} to our present setting, but this is also follows easily from the results of \cite{ChidaHsieh}; in particular, see \cite[Proposition 6.8]{ChidaHsieh}.
\end{proof}

Now we will reinterpret this pairing by calculating the Tate cohomology groups.

\begin{lemma}
We have an isomorphism
\[
\varinjlim_n \varinjlim_m \hat{H}^{-1}\bigl(G_n,\Sel(K_n,A_m)\bigr) \simeq \Sel(K_\infty, A)_\Gamma,
\]
where the subscript denotes $\Gamma$-coinvariants.
\end{lemma}
\begin{proof}
The proof is identical to the proof of \cite[Lemma 6.5]{Bertolini-Bordeaux}.
\end{proof}

We  define the \textit{universal norm submodule} of $\Sel(K, T).$ by
\[
US(K, T)= \bigcap_{n \geq 1}\mathrm{cores}_{K_n/K}\Sel(K_n, T).
\]

\begin{lemma}
We have an identification
\[
\varprojlim_n \varprojlim_m \hat{H}^0\bigl(G_n,\Sel(K_n,A_m)\bigr) \simeq \Sel(K, T) / US(K, T).
\]
\end{lemma}
\begin{proof}
The proof of \cite[Lemma 6.6]{Bertolini-Bordeaux} holds verbatim.
\end{proof}

Putting this all together, we obtain the following.

\begin{theorem}\label{thm:perfect-pairing}
There exists a perfect pairing
\[
\langle\! \langle - , - \rangle\! \rangle \colon \Sel(K, T) / US(K, T) \times \Sel(K_\infty, A)_\Gamma \rightarrow \Gamma \otimes_{\cO} \mathfrak{F} / \cO.
\]
\end{theorem}

As a consequence of this theorem, we obtain the following useful corollary.

\begin{corollary}\label{cor:no-finite-findex-if-torsion-free}
The $\Lambda$-module $\Sel(K_\infty,A)$ admits no proper $\Lambda$-submodule of finite index if and only if $\Sel(K, T) / US(K, T)$ is $\cO$-torsion-free.
\end{corollary}
\begin{proof}
See \cite[Corollary 6.2]{Bertolini-Bordeaux}.
\end{proof}

\subsection{Computation of universal norms}\label{subsec:universal-norm-com[utations}

In this section we will prove Theorem~\ref{thm:main}, which will follow from Corollary~\ref{cor:no-finite-findex-if-torsion-free} if we can show that $\Sel(K, T) / US(K, T)$ is torsion-free.

Recall that $R_n=R_{1,n}=(\cO/\varpi\cO)[G_n]$, that $\fA_n=\fA_{1,n}$, and that $A_1=  A[\varpi]$.

\begin{lemma} \label{lem:free-submodule-U}
If $\fA_n\ne0$, then $\Sel(K_n,A_1)$ admits  a free $R_n$-submodule $U_n$ such that $\fA_n\subset U_n$.
\end{lemma}
\begin{proof}
This is proved in exactly the same way as \cite[Lemma 5.5]{HLV}; one must simply ignore the $\pm$ symbols and replace $\fE$ by $\fA$. The only detail which will be necessary later is that $U_n$ is actually defined as an appropriate submodule of $\fA_{n+1}$ which is invariant under $\mathrm{Gal}(K_{n+1}/K_n)$.
\end{proof}

Let us now recall the definition of the Shafarevich--Tate groups in our setting, as well as the definition of relative Shafarevich--Tate groups.

\begin{definition} Let $n \in \N \cup \{ \infty \}$.
Write $\D(K_n,A)$ for the maximal divisible subgroup of $\Sel(K_n,A)$. Then we define the Shafarevich--Tate group of $f$ over $K_n$ by
\[
\Sh(K_n,A)=\Sel(K_n,A)/\D(K_n,A).
\]
\end{definition}

Consider the tautological exact sequence
\[
0 \rightarrow \D(K_n,A) \rightarrow \Sel(K_n,A) \rightarrow \Sh(K_n,A) \rightarrow 0.
\]
Since $\D(K_n,A)$ is divisible, we have $\D(K_n,A)/\varpi^m =0$ for any $m \geq 1$, so taking $\varpi^m$-torsion induces a short exact sequence
\[
0 \rightarrow \D(K_n,A)[\varpi^m] \rightarrow \Sel(K_n,A_m) \rightarrow \Sh(K_n,A)[\varpi^m] \rightarrow 0,
\]
where the middle term is identified via Theorem  \ref{thm:control}. We will simplify notation and instead write
\[
0 \rightarrow \D(K_n,A_m) \rightarrow \Sel(K_n,A_m) \rightarrow \Sh(K_n,A_m) \rightarrow 0.
\]
Recall from Theorem \ref{thm:control} that we have an injection $
\Sel(K_n,A_m) \hookrightarrow \Sel(K_{n+1},A_m)
$, hence we also have an injection $\D(K_n,A_m) \hookrightarrow \D(K_{n+1},A_m)$. Thus, the natural restriction map $H^1(K_n,A_m) \rightarrow H^1(K_{n+1},A_m)$ induces a map
\[
\Sh(K_n,A_m) \rightarrow \Sh(K_{n+1},A_m).
\]

\begin{definition}
We define the relative Tate-Shafarevich groups
\[
\Sh(K_{n+1}/K_n,A_m)=\ker \left( \Sh(K_n,A_m) \rightarrow \Sh(K_{n+1},A_m) \right)
\]
where the map is induced by restriction.
\end{definition}
Following Bertolini \cite{Bertolini-Bordeaux}, we impose the following assumption.
\\
\\
\textbf{(Sha)} For all $n \geq 0$, we assume $\Sh(K_{n+1}/K_n,A_1)=0$.
\\
\\
Under this assumption, we may show that each submodule $U_n$ defined in Lemma~\ref{lem:free-submodule-U} is contained in $\mathcal{D}(K_n,A_1)$. For $n \in \N$, write $\mathcal{G}_n$ for $\mathrm{Gal}(K_{n+1}/K_n)$.

\begin{proposition}\label{prop:un-in-dn}
Suppose $\Sh(K_{n+1}/K_n,A_1)=0$ and $\fA_n
\neq 0$. Then $U_n$ is contained in $\D(K_n,A_1)$.
\end{proposition}

\begin{proof}
Consider the commutative diagram with exact rows
\[
\begin{CD}
0 @>>> \D(K_n,A_1)   @>>> \Sel(K_n,A_1) @>>>\Sh(K_n,A_1) @>>> 0\\
@.    @VVV                   @VVV @VVV @.\\
0 @>>> \D(K_{n+1},A_1)^{\mathcal{G}_n}   @>>> \Sel(K_{n+1},A_1)^{\mathcal{G}_n} @>>>\Sh(K_{n+1},A_1)^{\mathcal{G}_n}
\end{CD}
\]
By Theorem \ref{thm:control}, the middle vertical map is an isomorphism, while the kernel of the right vertical map is $0$ by assumption, so by the snake lemma the left vertical map is also an isomorphism, so we have
\begin{equation}\label{eq:divisible-submodules}
\D(K_n,A_1) \simeq \D(K_{n+1},A_1)^{\mathcal{G}_n}.
\end{equation}
But $U_n$ is defined as a submodule of $\fA_{n+1}^{\mathcal{G}_n}$, so by Lemma \ref{lem:free-submodule-U}, $U_n$ is contained in the right-hand side of \eqref{eq:divisible-submodules}, hence also the left-hand side, which completes the proof.
\end{proof}

The utility of this last proposition comes from our ability to describe $\Sel(K,T)$ in terms of the modules $\D(K_n,A_m)$.

\begin{proposition}\label{prop:sel-is-tate}
We have an identification
\[
\Sel(K,T) = \varprojlim_m \D(K,A_m).
\]
\end{proposition}
\begin{proof}
For each $m \geq 1$ we have, by definition, a short exact sequence
\[
0 \rightarrow \D(K,A_m) \rightarrow \Sel(K,T_m) \rightarrow \Sh(K,A_m) \rightarrow 0
\]
where, for the middle term, we have remembered that $T_m \simeq A_m$ and that their Selmer groups agree at finite level. Each of these groups is finite, so projective limits will preserve the exact sequence, and the order of $\Sh(K,A_m)$ is bounded independent of $m$, so taking $\varprojlim_m$ proves the claim.
\end{proof}

The following proposition completes the proof of Theorem \ref{thm:main} via Corollary \ref{cor:no-finite-findex-if-torsion-free}.

\begin{proposition} Assume $(f,K,p)$ is admissible in the sense of $\S\ref{sec:notation}$ and that hypothesis \textbf{(Sha)} holds. Then $\Sel(K,T)/US(K,T)$ is a torsion-free $\cO$-module.
\end{proposition}

\begin{proof}
We begin by noting that $\Sel(K,T)$ is a torsion-free $\Lambda$-module of rank $1$; this is \cite[Proposition 3.4.3]{howard} when $k=2$ and \cite[Theorem 3.5]{LongoVigni} when $k \geq 4$.

The same theorems also tell us that $\mathrm{rank}_\Lambda \X(f)=1$, so arguing as in \cite[$\S$3.2]{Bertolini-Compositio}, we have
\begin{equation}\label{eq:common-rank}
\mathrm{rank}_\cO US(K_\infty, T) = \rank_\Lambda \hat{S}(K_\infty, T) = \rank_\Lambda \X(f) = 1.
\end{equation}
Thus, if we can show that $US(K,T)$ contains a nontrivial element of $\Sel(K,T)$ not divisible by $\varpi$, then $US(K,T) \simeq \cO$ and $\Sel(K,T)/US(K,T)$ is a torsion-free $\cO$-module as desired.

Let us write $T\D_n = \varprojlim_m \D(K_n,A_m)$ for the Tate module of $\D(K_n,A)$. Then by Proposition \ref{prop:un-in-dn}, we have an inclusion
\[
U_n \subset \D(K_n,A_1) \simeq T\D_n / \varpi T\D_n.
\]
Recall that $U_n$ is a free $R_n$-module. Let $\tilde U_n$ be a free $\Zp[G_n]$-submodule of $T\D_n$ of rank one lifting $U_n$ modulo $\varpi$, generated by an element $v_n$. Then the fact that $\tilde U_n$ is free implies that $\cor_{K_n/K}(v_n)$ is not divisible by $\varpi$. On the other hand, Lemma~\ref{prop:sel-is-tate} tells us that
\[
\cor_{K_n/K}(v_n)\in T\D_0=US(K,T).
\]
By compactness, we may find a subsequence $\bigl(\cor_{K_n/K}(v_{n_i})\bigr)_{i\ge 0}$ converging to an element of  $\Sel(K,T)$ that lies in $US(K,T)$ and is not divisible by $\varpi$, as required.
\end{proof}
\bibliographystyle{amsalpha}
\bibliography{references}
\end{document}